\documentclass[final]{siamltex}

\usepackage{amsmath, amssymb, amsfonts, amsthm} 
\usepackage{graphicx,color,caption,subcaption}
\usepackage{diagbox}
\usepackage[ruled]{algorithm2e}
\usepackage{algorithmic}
\usepackage{enumitem}
\usepackage{multirow}
\usepackage{epstopdf}
\usepackage[capitalize]{cleveref}
\usepackage{placeins}
\usepackage{comment}
\usepackage{enumitem}

\newcommand{\Tr}{\mathrm{Tr}}

\newcommand{\mc}[1]{\mathcal{#1}}

\newcommand{\ud}{\,d}

\newcommand{\bra}[1]{\langle#1\vert}
\newcommand{\ket}[1]{\vert#1\rangle}

\newcommand{\Or}{\mathcal{O}}

\newcommand{\I}{\imath}
\newcommand{\RR}{\mathbb{R}}

\newcommand{\re}{r_{\ve}}
\newcommand{\dre}{\dot{r}_{\ve}}
\newcommand{\ddre}{\ddot{r}_{\ve}}

\newcommand{\pe}{p_{\ve}}
\newcommand{\dpe}{\dot{p}_{\ve}}

\newcommand{\xe}{x_{\ve}}
\newcommand{\dxe}{\dot{x}_{\ve}}
\newcommand{\ddxe}{\ddot{x}_{\ve}}

\newcommand{\ye}{y_{\ve}}

\newcommand{\ret}{r_{\ve}}

\newcommand{\pet}{p_{\ve}}

\newcommand{\yet}{y_{\ve}}
\newcommand{\dyet}{\dot{y}_{\ve}}

\newcommand{\yett}{\widetilde{y}_{\ve}}
\newcommand{\dyett}{\dot{\widetilde{y}}_{\ve}}
\newcommand{\yt}{\widetilde{y}}
\newcommand{\dyt}{\dot{\widetilde{y}}}

\numberwithin{equation}{section}
\numberwithin{figure}{section}
\newtheorem{thm}{\protect\theoremname}
\newtheorem{lem}[thm]{\protect\lemmaname}
\newtheorem{rem}[thm]{\protect\remarkname}
\newtheorem{prop}[thm]{\protect\propositionname}

\newtheorem{assumption}[thm]{\protect\assumptionname}
\newtheorem{defn}[thm]{\protect\definitionname}

\providecommand{\definitionname}{Definition}
\providecommand{\assumptionname}{Assumption}
\providecommand{\corollaryname}{Corollary}
\providecommand{\lemmaname}{Lemma}
\providecommand{\propositionname}{Proposition}
\providecommand{\remarkname}{Remark}
\providecommand{\theoremname}{Theorem}

\title{Towards sharp error analysis of \\extended Lagrangian molecular dynamics}

\author{Dong An\thanks{Department of Mathematics, University of California, Berkeley,
CA 94720. Email: \texttt{dong\_an@berkeley.edu}}
\and
Lin Lin\thanks{Department of Mathematics, University of California, Berkeley, and Computational Research Division, Lawrence Berkeley National Laboratory, Berkeley, CA 94720. Email: \texttt{linlin@math.berkeley.edu}}
\and
Michael Lindsey\thanks{Department of Mathematics, Courant Institute of Mathematical Sciences, New York University, New York, NY 10012.  Email:
\texttt{michael.lindsey@cims.nyu.edu}}
}

\begin{document}

\global\long\def\ve{\varepsilon}
\global\long\def\R{\mathbb{R}}
\global\long\def\Rn{\mathbb{R}^{n}}
\global\long\def\Rd{\mathbb{R}^{d}}
\global\long\def\E{\mathbb{E}}
\global\long\def\P{\mathbb{P}}
\global\long\def\bx{\mathbf{x}}
\global\long\def\vp{\varphi}
\global\long\def\ra{\rightarrow}
\global\long\def\smooth{C^{\infty}}
\global\long\def\Tr{\mathrm{Tr}}
\global\long\def\bra#1{\left\langle #1\right|}
\global\long\def\ket#1{\left|#1\right\rangle }

\maketitle

\begin{abstract}
The extended Lagrangian molecular dynamics (XLMD) method provides a useful framework for reducing the computational cost of a class of molecular dynamics simulations with constrained latent variables. The XLMD method relaxes the constraints by introducing a fictitious mass $\varepsilon$ for the latent variables, solving a set of singularly perturbed ordinary differential equations. While favorable numerical performance of XLMD has been demonstrated in several different contexts in the past decade, mathematical analysis of the method remains scarce. We propose the first error analysis of the XLMD method in the context of a classical polarizable force field model. While the dynamics with respect to the atomic degrees of freedom are general and nonlinear, the key mathematical simplification
of the polarizable force field model is that the constraints on the latent variables are given by a linear system of equations. We prove that when the initial value of the latent variables is \emph{compatible} in a sense that we define, XLMD converges as the fictitious mass $\varepsilon$ is made small with $\mathcal{O}(\varepsilon)$ error for the atomic degrees of freedom and with $\mathcal{O}(\sqrt{\varepsilon})$ error for the latent variables, when the dimension of the latent variable $d'$ is 1. Furthermore, when the initial value of the latent variables is improved to be \emph{optimally compatible} in a certain sense, we prove that the convergence rate can be improved to $\mathcal{O}(\varepsilon)$ for the latent variables as well. Numerical results verify that both estimates are sharp not only for $d' =1$, but also for arbitrary $d'$. In the setting of general $d'$, we do obtain convergence, but with the non-sharp rate of $\mc{O}(\sqrt{\ve})$ for both the atomic and latent variables.
\end{abstract}

\section{Introduction}

Molecular dynamics simulation of a system with many atoms often requires solving a set of self-consistent equations for a set of latent variables in order to obtain the potential energy and the atomic force. Examples include \textit{ab initio} molecular dynamics (AIMD) ~\cite{CarParrinello1985,Martin2008,MarxHutter2009},
reactive force field (ReaxFF) \cite{DuinDasguptaLorantEtAl2001}, polarizable force field (PFF)~\cite{PonderWuRenEtAl2010,AlbaughDemerdashHead-Gordon2015}, etc. If such self-consistent equations are to be solved fully self-consistently, then the computational cost can be very high. On the other hand, inaccurate solution of these equations often leads to noticeable energy drifts and inaccurate (or even unstable) results for long-time simulation. Various attempts have been made to tackle this problem in the past few decades across different fields. For example, in AIMD, the latent variables are the electronic wavefunctions, and the self-consistent equations are the Kohn-Sham equations \cite{KohnSham1965}. The seminal work of the Car-Parrinello molecular dynamics (CPMD) \cite{CarParrinello1985} constructs an extended Lagrangian (XL) that introduces a fictitious mass $\varepsilon$ for the electronic wavefunctions. Instead of solving the Kohn-Sham equations self-consistently, CPMD propagates the electronic wavefunctions efficiently following fictitious Newtonian dynamics, similar to those for the atoms. In the past decade, Niklasson and co-workers \cite{Niklasson2006,Niklasson2008,Niklasson2012} have introduced another type of extended Lagrangian molecular dynamics (XLMD). XLMD also associates a fictitious mass to electronic degrees of freedom but in a way that is more flexible than CPMD. In classical simulation with a PFF model, the latent variable is the induced dipole field. Head-Gordon and co-workers have generalized the idea of XLMD to accelerate the PFF simulation\footnote{The name of the method is ``inertial extended Lagrangian with zero-self-consistent field iteration'' (iEL/0-SCF).}~\cite{AlbaughDemerdashHead-Gordon2015,AlbaughHead-Gordon2017,AlbaughNiklassonHead-Gordon2017}. It was found that in a number of cases, the number of self-consistent field iterations can be even set to zero, while the dynamics remains accurate and stable.

Despite the practical success of extended Lagrangian type methods, mathematical analysis on this topic remains scarce. In the context of AIMD, as the fictitious mass $\varepsilon\to 0$, the convergence of CPMD has been established \cite{PastoreSmargiassiBuda1991,BornemannSchutte1998} for insulating systems with an $\Or(\sqrt{\varepsilon})$ convergence rate. In \cite{LinLuShao2014}, the effectiveness of the XLMD method was studied in the linear response regime (with respect to both atomic and latent degrees of freedom). It was found that XLMD can be accurate for both insulating and metallic systems and that the convergence rate can be improved to be $\Or(\varepsilon)$. However, such analysis was based on explicit diagonalization and perturbation theory, which is difficult to generalize to nonlinear systems.

In this paper, we give the first convergence analysis of the  XLMD method in the context of the PFF model. Compared to the general setup of XLMD, the key mathematical simplification of the PFF model is that self-consistent equations are linear with respect to the latent variables. Meanwhile, the dynamics for the atomic degrees of freedom are still general and nonlinear. The convergence of the XLMD method crucially depends on the initial value of the latent variables. We prove that when the initial condition of the latent variables is \emph{compatible} (see \cref{def:initial}), the XLMD method converges, and the convergence rate is $\Or(\sqrt{\varepsilon})$ for both the atomic and the latent variables. When the dimension $d'$ of the latent variable is one (though the dimension of the atomic degrees of freedom can be arbitrary), we prove that the error for the atomic variables can be improved to $\Or(\varepsilon)$. Numerical results verify that the rates of $\Or(\varepsilon)$ and $\Or(\sqrt{\varepsilon})$, respectively, for the atomic and latent variables are sharp for arbitrary $d'$. Meanwhile, the initial condition of the latent variables can be improved to be \emph{optimally compatible} in a sense that we define. In such a case, numerical results indicate that the convergence rate of the atomic degrees of freedom remains $\Or(\varepsilon)$, but the convergence rate of the latent variables improves to $\Or(\varepsilon)$. We prove that when $d'=1$, the convergence rate of the latent variables is indeed $\Or(\varepsilon)$. Our sharp proofs in the $d'=1$ case rely on certain special commutative properties which allow for detailed analysis of the one-dimensional harmonic oscillator with time-dependent mass.  Hence the generalization of our sharp results to higher-dimensional systems may require nontrivial modifications.

Our analysis applies to XLMD method, which is specified by a set of deterministic ordinary differential equations. We remark that a variant of the XLMD method applies a thermostat to the auxiliary variables \cite{Niklasson2009}, where a stochastic force term is introduced to balance the possible accumulation of numerical errors. It has been found numerically that the kinetic energy of the latent variable may accumulate in a long-time ReaxFF simulations \cite{Tanetal2020}. Mathematically, the introduction of a stochastic thermostat effectively enforces ergodicity of the latent variables in the limit $\varepsilon\to 0$ and can simplify the analysis of the method. For the PFF model, the stochastic extended Lagrangian molecular dynamics (S-XLMD) method \cite{AnChengHead-GordonEtAl2019} can converge with arbitrary initial condition for the latent variable. However, the convergence rate for the atomic degrees of freedom can only be $\Or(\sqrt{\varepsilon})$, which is weaker than that of the XLMD method with compatible initial conditions, at least in the context of short time simulation.

The rest of the paper is organized as follows. In \cref{sec:setup}, we discuss the mathematical setting of the XLMD method for the PFF model and state the main results. The details of the first part of the main result (error analysis when the dimension of the latent variable is arbitrary) are given in \cref{sec:analysis_general}, while those of the second part (error analysis when the latent variable is one-dimensional) are given in \cref{sec:analysis_1d}. In fact \cref{sec:analysis_1d} bootstraps from the error bound proved in the preceding \cref{sec:analysis_general}, hence cannot be read independently. We validate the error analysis with numerical results in \cref{sec:numer}.

\section{Problem setup and main results}\label{sec:setup}

In a simplified mathematical setting, the problem can be stated as
follows. Let $r\in \R^{d}$ be the collection of atomic positions, and
$x\in \R^{d'}$ be the latent variable (the induced dipole in the polarizable force field model). 
Let $F(r)$ be an external force involving only the atomic positions. Of particular interest is the case of a conservative force field, i.e., the case in which $F$ is determined by an external potential field $U(r)$ via $F(r) = - \frac{\partial U}{\partial r}(r)$.\footnote{In fact our main results do not directly depend on any assumption of a conservative force, though we will use such an assumption to guarantee certain \emph{a priori} bounds needed for our analysis.}
Let $Q(r,x)$ be the interaction energy between the atoms and the latent
variable. In the polarizable force field model, $Q(r,x)$ is a quadratic function in terms of $x$: 
\begin{equation}\label{eqn:Hamiltonian_Quadratic}
    Q(r,x) = \frac{1}{2}x^{\top} A(r)x - b(r)^{\top}x.
\end{equation}
Here the mappings $b:\R^{d}\ra\R^{d'}$ and $A:\R^{d}\ra\mathcal{S}_{++}^{d'}$ are smooth, where $\mathcal{S}_{++}^{d'}$ denotes the set of real symmetric positive definite $d' \times d'$ matrices. For a given $r$, the latent variable $x$ that minimizes the interaction energy is determined by the equation $\frac{\partial Q}{\partial x}(r,x)=0$, i.e., by the linear  system of equations 
\[
A(r) x = b(r).
\] 
Since $A(r)$ is positive definite, it is in particular invertible, and the solution $x(r)$ is unique for all $r$.

Then the polarizable force field simulation requires the solution of the following system of differential-algebraic equations 
\begin{subequations}\label{eqn:General_Dynamics}
\begin{align}
    \ddot{r}_{\star}(t) &= F(r_{\star}(t)) - \frac{\partial Q}{\partial
    r}(r_{\star}(t),x_{\star}(t)), \label{eqn:General_Dynamics_a}\\
    0 &= b(r_{\star}(t))-A(r_{\star}(t))x_{\star}(t), \label{eqn:General_Dynamics_b}
\end{align}
\end{subequations}
for all $0 \leq t \leq t_f$, subject to certain initial conditions
$r_{\star}(0),\dot{r}_{\star}(0)$.  Here the subscript $\star$ is used to
indicate the exact solution of Eq.~\eqref{eqn:General_Dynamics}.  Note
that the initial condition for $x$ need not be specified, since it can be
determined from $r_{\star}(0)$ through Eq.~\eqref{eqn:General_Dynamics_b}.

In molecular dynamics simulation, we are generally more interested in the accuracy of the atomic trajectory $r(t)$ than that of the latent variable $x(t)$.
Nonetheless, the solution of the linear 
system~\eqref{eqn:General_Dynamics_b} (typically by iterative methods for large systems of interest) is often the most costly step in a polarizable force field simulation.  Following the setup of \eqref{eqn:General_Dynamics}, the XLMD method can be introduced as follows. We first define an extended Lagrangian as
\begin{equation}
  L_{\ve} = \frac12 |\dre|^2 + \frac{\varepsilon }{2} |\dxe|^2 - U(\re) -
  Q(\re,\xe),
  \label{eqn:extended_lagrangian}
\end{equation}
where $\varepsilon > 0$ can be interpreted as a fictitious mass for the latent variable $x_{\varepsilon}$. The corresponding Euler-Lagrange
equations are
\begin{subequations}\label{eqn:General_Deterministic}
\begin{align}
    \ddre &= F(\re) - \frac{\partial Q}{\partial r}(\re,\xe), \label{eqn:General_Deterministic_a}\\
    \varepsilon \ddxe &= -\frac{\partial Q}{\partial x}(\re,\xe)=b(\re)-A(\re) \xe. \label{eqn:General_Deterministic_b}
\end{align}
\end{subequations}
When the force $F(r)$ is conservative, Eq.~\eqref{eqn:General_Deterministic} is a singularly
perturbed Hamiltonian system, and it 
can be discretized with symplectic or time-reversible integrators to
obtain long-time stability~\cite{HairerLubichWanner2006}.  
Note that the value of $\sqrt{\ve}$ provides an upper
bound for the time step of second order numerical
integrators (up to a multiplicative constant) ~\cite{Niklasson2008,AlbaughNiklassonHead-Gordon2017,AlbaughHead-Gordon2017}. Therefore it is desirable choose $\ve$ to be not too small in practice.
Although Eq.~\eqref{eqn:General_Deterministic} introduces a systematic
error in terms of $\ve$, when $\ve$ is chosen properly the XLMD
method often outperforms the original (discretized) dynamics in terms of
efficiency and long-time stability while still maintaining sufficient accuracy
for the atomic trajectory.

Note that initial conditions for $\xe$ and $\dxe$ are needed 
for~\eqref{eqn:General_Deterministic}. 
A natural choice for $\xe(0)$ is 
\begin{equation}\label{eqn:compatible_initial_value}
    \xe(0) = x_{\star}(0) = A(r_{\star}(0))^{-1}b(r_{\star}(0)) {,}
\end{equation}
which requires the linear system to be solved very accurately at the beginning.
Moreover, a natural choice for $\dxe(0)$ can also be derived as
\begin{equation}\label{eqn:optimally_compatible_initial_value}
\begin{split}
     \dxe(0) = \dot{x}_{\star}(0)
     = &-A(r_{\star}(0))^{-1}\left[\sum_{k=1}^{d} \dot{r}_{\star,k}(0)\frac{\partial A}{\partial r_k}(r_{\star}(0))\right] A(r_{\star}(0))^{-1}b(r_{\star}(0))\\
     &\quad +A(r_{\star}(0))^{-1}\left[\sum_{k=1}^{d} \dot{r}_{\star,k}(0)\frac{\partial b}{\partial r_k}(r_{\star}(0))\right]{,}
\end{split}
\end{equation}
where the second equality can be obtained by differentiating Eq.~\eqref{eqn:General_Dynamics_b} and then letting $t = 0$. \\
\begin{defn}[Optimally compatible and compatible initial condition]\label{def:initial}
We say that we have chosen the \emph{optimally compatible} initial condition
if $\xe(0)$ and $\dxe(0)$ are specified by \cref{eqn:compatible_initial_value,eqn:optimally_compatible_initial_value}. If $\xe(0)$ satisfies \cref{eqn:compatible_initial_value} but $\dxe(0)$ is only given in a way that is uniformly bounded with respect to $\ve$, we say that we have chosen a \emph{compatible} initial condition.\\
\end{defn} 

As we will see later, choosing a compatible initial condition is essential 
for the convergence of XLMD. In turn optimal compatibility can ensure even better convergence as $\varepsilon \ra 0$ for the latent variable.

Consider a fixed time interval $[0,t_f]$ with $t_f = \Or(1)$ as $\varepsilon \ra 0$. Throughout the paper $C$ will denote a sufficiently large constant that is independent of $\varepsilon$ (though perhaps dependent on other aspects of the problem specification, e.g., the potential $U$).  Now we enumerate several technical assumptions that we need for our results.\\
\begin{assumption}\label{assump:tech}
We make the following assumptions.
\begin{enumerate}[label={(\roman*)}]
    \item $A:\R^{d}\ra\mathcal{S}_{++}^{d'}$ is a $C^{3}$ map, and 
    there exists $C>0$ such that $A(r)\succeq C^{-1}$ for all $r\in\R^{d}$.
    \item $b:\R^{d}\ra\R^{d'}$ is a $C^{3}$ map. 
    \item $F:\R^{d}\ra\R^{d}$ is a $C^{2}$ map.
    \item All the initial values for $(r_{\star},p_{\star})$ and $(\re,\pe,\xe,\dxe)$ are 
    bounded independently of $\varepsilon$, with 
    $r_{\star}(0) = \re(0), p_{\star}(0) = \pe(0)$.
    \item There exist unique solutions for the systems~\eqref{eqn:General_Dynamics} and~\eqref{eqn:General_Deterministic} on $[0,t_f]$. Furthermore, the solutions $r_{\star}, \re, \xe$ are $C^3$ functions and satisfy \emph{a priori} bounds  
    $\left|\frac{d^{k}r_{\star}}{dt^{k}}\right|,\left|\frac{d^{k}r_{\ve}}{dt^{k}}\right|\leq C$ for $k=0,1,2$, and $|\xe|, \sqrt{\varepsilon}|\dxe| \leq C$, where $C$ is a constant independent of $\ve$. \\
\end{enumerate}
\end{assumption}

The first assumption that $A$ is globally positive definite is physical and satisfied in the polarizable force field model \cite{AlbaughNiklassonHead-Gordon2017}. 
The last assumption assumes the global existence and uniqueness 
of the solutions of both the exact MD~\eqref{eqn:General_Dynamics} 
and the XLMD~\eqref{eqn:General_Deterministic} with 
\emph{a priori} estimates that are important for 
our analysis. 
If $F$ is obtained as the gradient of a potential $U$
bounded from below and $b$ is bounded, then the
last assumption follows from the preceding assumptions
(i), (ii), (iii), and (iv). 
We summarize this remark in the following proposition.\\
\begin{prop}\label{prop:solution}
    Consider the conservative force
    $F = - \frac{\partial U}{\partial r}$, where $U: \R^d\ra\R$ is a $C^{2}$ map bounded from below. 
    Assume moreover that $b$ is bounded. 
    Then in Assumption~\ref{assump:tech}, statements (i), (ii), (iii), and (iv) imply statement (v). 
\end{prop}
\vspace{5mm}
The proof is given in the Appendix. Now we may state our main result.\\

\begin{thm}\label{thm:main_theorem}
    Let $(r_{\star},p_{\star})$ solve the exact MD in \cref{eqn:General_Dynamics}
    and $(\re,\pe,\xe,\dxe)$ solve the XLMD in \cref{eqn:General_Deterministic},
    and assume that the initial condition for the latent variable is compatible according to \cref{def:initial}.  Then under \cref{assump:tech},
\begin{enumerate}[label={(\roman*)}]
    \item for general $d'$, there exists an $\ve$-independent constant $C > 0$ such that 
    \begin{equation}\label{eqn:coarse_estimate} 
        \lvert \ret(t) -
  r_{\star}(t)\rvert, \ \lvert \pet(t) - p_{\star}(t)\rvert  \leq C \ve^{1/2}
    \end{equation}
    for all $t\in [0,t_f]$.
    Under these conditions, we also have that
        \begin{equation}\label{eqn:coarse_estimate2} 
        \lvert x_\ve(t) -
  x_{\star}(t)\rvert \leq C \ve^{1/2}
    \end{equation}
     for all $t\in [0,t_f]$.
    
    \item if the latent variable has dimension $d' = 1$, 
    then we have a sharp estimate 
    \begin{equation}\label{eqn:sharp_estimate}
        \lvert \ret(t) -
  r_{\star}(t)\rvert, \  \lvert \pet(t) - p_{\star}(t)\rvert  \leq C \ve
  \end{equation}
  for all $t\in [0,t_f]$.
  Under these conditions, we have that~\eqref{eqn:coarse_estimate2} holds in general, but if the initial condition is moreover optimally compatible, then we have the tighter estimate
  \begin{equation}
       \lvert x_\ve (t) - x_{\star}(t)\rvert  \leq  C \ve.
    \end{equation}
    for all $t\in [0,t_f]$.\\
\end{enumerate}
\end{thm}

The proof of (i) and (ii) of \cref{thm:main_theorem} will be given in \cref{sec:analysis_general} and \cref{sec:analysis_1d}, respectively. Numerical results in \cref{sec:numer} confirm that the estimate in  \cref{eqn:sharp_estimate} is sharp. They also indicate that the estimates in (ii) should in fact hold for general $d'$, but a sharp result for general $d'$ is beyond the framework of our analysis.

\section{Error analysis for any $d'$}\label{sec:analysis_general}

We first briefly sketch the main idea for proving \cref{eqn:coarse_estimate}. It is helpful to take a more abstract perspective to see how we will
proceed from our understanding of the dynamics of the $x$ variable to that of the $r$ variable. By defining 
\begin{equation}\label{eqn:def_G}
    G(r,x):=F(r)-\left[\frac{1}{2}x^{\top}\frac{\partial A}{\partial r}(r)x-\frac{\partial b^{\top}}{\partial r}(r)x\right],
\end{equation}
and plugging Eq.~\eqref{eqn:General_Dynamics_b} into Eq.~\eqref{eqn:General_Dynamics_a}, 
we can rewrite the exact MD in terms of $(r,p)$ as
\begin{subequations}\label{eqn:MD_rp}
\begin{align}
    \dot{r}_{\star} &= p_{\star}, \label{eqn:MD_rp_a}\\
    \dot{p}_{\star} &= G(r_{\star},A(r_{\star})^{-1}b(r_{\star})). \label{eqn:MD_rp_b}
\end{align}
\end{subequations}
The XLMD reads as 
\begin{subequations}\label{eqn:0SCF_rp}
\begin{align}
    \dre &= \pe, \label{eqn:0SCF_rp_a}\\
    \dpe &= G(\re,\xe) \label{eqn:0SCF_rp_b}\\
    \varepsilon \ddxe &= b(\re)-A(\re) \xe. \label{eqn:0SCF_rp_c}
\end{align}
\end{subequations}
Since XLMD only introduces a singular perturbation on the latent variable, 
it is reasonable to expect that $\xe$ is close to $A^{-1}(\re)b(\re)$ up to a small perturbation. 
If so, intuitively, $(\re,\pe)$ is governed by an ODE 
which is only a small perturbation of~\eqref{eqn:MD_rp}. Given the same initial value for 
$(r_{\star},p_{\star})$ and $(\re,\pe)$, this implies that 
$(\re,\pe)$ is also a small perturbation of $(r_{\star},p_{\star})$. 

To prove $\xe$ is indeed a small perturbation of $A^{-1}(\re)b(\re)$, 
it is useful to
think of the trajectory $r_{\ve}$ as being fixed and then study the behavior
of $x_{\ve}$ according to Eq.~\eqref{eqn:General_Deterministic_b}, 
which can be viewed as a linear inhomogeneous 
ODE with time-dependent coefficients. We may then use variation of parameters to prove \cref{eqn:coarse_estimate}.


Since we expect that $x_{\ve}\approx A^{-1}(r_{\ve})b(r_{\ve})$, we define the new residual variable
\[
y_{\ve}:=x_{\ve}-A(r_{\ve})^{-1}b(r_{\ve}).
\]
From Eq.~\eqref{eqn:0SCF_rp_c} 
the evolution of $y_{\ve}$ is given by
\begin{equation}\label{eqn:ODE_residue}
    \ve\ddot{y}_{\ve}=-A(\re)y_{\ve}+\ve\psi_{\ve},
\end{equation}
 where
\begin{equation}
    \psi_{\ve}:=-\frac{d^{2}}{dt^{2}}\left[A(r_{\ve})^{-1}b(r_{\ve})\right].
\end{equation}
By Assumption~\ref{assump:tech}, there exists $C$
such that $\vert\psi_{\ve}\vert\leq C$, uniformly in $\ve$.
By the definition of $\ye$, 
the initial conditions for $y_{\ve}$ and $\dot{y}_{\ve}$ are given by 
\begin{equation}
    \yet(0) = 0,\ \dyet(0)=z_{0},
\end{equation}
where $z_0$ is uniformly bounded in $\ve$. Note that by construction $z_0 = 0$ in the optimally compatible case.




It is natural to approach the inhomogeneous linear system of ODEs of~\eqref{eqn:ODE_residue} via Duhamel's principle, which suggests to study the corresponding homogeneous linear system for all starting times $s\in [0,t_f]$. To wit, now consider the homogeneous equation
\begin{subequations}\label{eqn:homogeneous}
\begin{align}
    &\ve\ddot{\yett}=-A(\re)\yett, \label{eqn:homogeneous_a}\\
    &\yett(s) = \eta_0, \ \dyett(s)=\xi_{0}, \label{eqn:homogeneous_b}
\end{align}
\end{subequations}
where the starting time $s$ 
and initial values $\eta_0$, $\xi_{0}$ are arbitrary. 
We define the flow
map for the homogeneous system~\eqref{eqn:homogeneous} by
\begin{equation}
    \Phi_{\ve}^{s,t}(\eta_0, \xi_{0})=\left(\begin{array}{c}
\yett(t)\\
\dyett(t)
\end{array}\right)
\end{equation}
 for $t\geq s$, where $\yett$ is the solution of~\eqref{eqn:homogeneous}.
Define
\begin{equation}
    K_{\ve}(t)=A(\re(t))^{1/2},
\end{equation}
where the matrix square root operation is well defined due to \cref{assump:tech}(i). Also define
$U_{\ve,+}^s(t)$ to be the solution of the following 
initial value problem
\begin{equation}
    \dot{U}_{\ve,+}^s(t)=\I\ve^{-1/2}K_{\ve}(t)U_{\ve,+}^s(t),\quad U_{\ve,+}^s(s)=I_{d}.
\end{equation}
In `physicists' notation, one writes
\begin{equation}
    U_{\ve,+}^s(t) := \mathcal{T}e^{\I\ve^{-1/2}\int_{s}^{t}K_{\ve}(t')\,dt'},
\end{equation}
where $\mathcal{T}$ is the `time ordering operator'. Note that this is merely a notation and can be ignored in favor of the formal definition.

Similarly define 
\begin{equation}
    U_{\ve,-}^s(t)=\mathcal{T}e^{-\I\ve^{-1/2}\int_{s}^{t}K_{\ve}(t')\,dt'},
\end{equation}
 i.e., $U_{\ve,-}^s(t)$ solves
\begin{equation}
    \dot{U}_{\ve,-}^s(t)=-\I\ve^{-1/2}K_{\ve}(t)U_{\ve,-}^s(t),\ \ U_{\ve,-}^s(s)=I_{d'}.
\end{equation} \\
By construction $U_{\ve,+}^s$ and $U_{\ve,-}^s$ are unitary matrices for all $t$.\\


\begin{lem}\label{lem:homogeneous}
Let $\Phi_{\ve}^{s,t}(\eta_0, \xi_{0})$ be the flow map of the homogeneous
system~\eqref{eqn:homogeneous}. Then 
\begin{enumerate}[label={(\roman*)}]
\item $\Phi_{\ve}^{s,t}(\eta_0,\xi_0)$ can be written in the form 
\[
    \Phi_{\ve}^{s,t}(\eta_0, \xi_{0}) = 
    \left(\begin{array}{c}
        U_{\ve,+}^s(t)c_{\ve,+}^s(t)+U_{\ve,-}^s(t)c_{\ve,-}^s(t)   \\
        \I\ve^{-1/2}K_{\ve}(t)\left[U_{\ve,+}^s(t)c_{\ve,+}^s(t)-U_{\ve,-}^s(t)c_{\ve,-}^s(t)\right]   
    \end{array}\right), 
\]

where $c_{\ve,+}^s(t)$ and $c_{\ve,-}^s(t)$ follow the estimates 
\begin{equation}
    \vert c_{\ve,+}^s(t)\vert,\vert c_{\ve,-}^s(t)\vert\leq C(\vert\eta_{0}\vert + \ve^{1/2}\vert\xi_{0}\vert).
\end{equation}
Here $C$ is independent of $\ve$, $\eta_0$ and $\xi_{0}$. 

\item $\Phi$ follows the estimate
\begin{equation}
    \Phi_{\ve}^{s,t}(\eta_0, \xi_{0}) = \left(\begin{array}{c}
        \Or(\vert\eta_{0}\vert + \ve^{1/2}\vert\xi_{0}\vert)   \\
        \Or(\ve^{-1/2}\vert\eta_{0}\vert + \vert\xi_{0}\vert) 
    \end{array}\right) {.}\\
\end{equation}
\end{enumerate}
\end{lem}
 
\begin{proof}
For notational simplicity we will omit the dependence
on $\ve$ from the subscripts and the explicit time dependence on $s,t$. Consider the ansatz
\begin{equation}\label{eqn:y}
\yt=U_+c_+ + U_-c_-,
\end{equation}
where $c_+$ and $c_-$ are to be determined. Following the idea of variation of parameters, we assume\begin{equation}
   U_+\dot{c}_+ + U_-\dot{c}_- = 0{.}
\end{equation}
Therefore we have 
\begin{equation}\label{eqn:py}
    \dyt = \dot{U}_+ c_+ + \dot{U}_-c_- 
    = \I \varepsilon^{-1/2} (KU_{+}c_+ - KU_{-}c_-) {,}
\end{equation}
and
\begin{align*}
     \ddot{\widetilde{y}} &= -\varepsilon^{-1} (K^2U_{+}c_+ + K^2U_{-}c_-) + 
    \I\varepsilon^{-1/2}(\dot{K}U_{+}c_+ - \dot{K}U_{-}c_-) 
    + \I\varepsilon^{-1/2} (KU_{+}\dot{c}_+ - KU_{-}\dot{c}_-) \\
    &= -\varepsilon^{-1} A\yt + 
    \I\varepsilon^{-1/2}(\dot{K}U_{+}c_+ - \dot{K}U_{-}c_-) 
    + \I\varepsilon^{-1/2} (KU_{+}\dot{c}_+ - KU_{-}\dot{c}_-) {.}
\end{align*}
Compare with the homogeneous ODE~\eqref{eqn:homogeneous}, 
\begin{equation}
    \I\varepsilon^{-1/2}(\dot{K}U_{+}c_+ - \dot{K}U_{-}c_-) 
    + \I\varepsilon^{-1/2} (KU_{+}\dot{c}_+ - KU_{-}\dot{c}_-) = 0 {.}
\end{equation}
Therefore we obtain an ODE system of $c_+$ and $c_-$, 
\begin{subequations}
\begin{align}
    &U_+\dot{c}_+ + U_-\dot{c}_- = 0 \\
    &\dot{K}U_{+}c_+ - \dot{K}U_{-}c_- + KU_{+}\dot{c}_+ - KU_{-}\dot{c}_- = 0 {,}
\end{align}
\end{subequations}
or equivalently, 
\begin{subequations}\label{eqn:ODE_ab}
\begin{align}
    &\dot{c}_+ = -\frac{1}{2}U_{+}^{-1}K^{-1}\dot{K}U_{+}c_+ 
      + \frac{1}{2}U_{+}^{-1}K^{-1}\dot{K}U_{-}c_- \\
      &\dot{c}_- = \frac{1}{2}U_{-}^{-1}K^{-1}\dot{K}U_{+}c_+ - 
      \frac{1}{2}U_{-}^{-1}K^{-1}\dot{K}U_{-}c_- {.}
\end{align}
\end{subequations}
Here all the matrices in this ODE system are uniformly bounded. 
Specifically, $U_{+}$ and $U_{-}$ are unitary matrices, 
$K^{-1}$ is bounded due to Assumption~\ref{assump:tech}
(and in particular \emph{a priori} bounds for $\re$), and 
$\dot{K}$ is bounded by way of our \emph{a priori} bounds for $\re$ and $\dot{\re}$. 
Then by Gr{\"o}nwall's inequality, 
we have the bounds
\begin{equation}\label{eqn:bounds_c}
    |c_+(t)| \leq C(|c_+(s)|+|c_+(s)|), \quad |c_-(t)| \leq C(|c_-(s)|+|c_-(s)|) {,}
\end{equation}
i.e., $c_+$ and $c_-$ can be totally controlled by their initial values via a constant $C$, independent of $\ve$ and the initial values. (Recall that $t_f = O(1)$ in $\ve$, though our constant would grow exponentially in the final time $t_f$ if it were treated as an independent variable.)
To bound initial values, let $t = s$ in Eq.~\eqref{eqn:y} and~\eqref{eqn:py}:
\begin{align*}
    &c_+(s) + c_-(s) = \eta_0 \\
    &c_+(s) - c_-(s) = -\I\varepsilon^{1/2}K(s)^{-1}\xi_0 {,}
\end{align*}
and
\begin{equation}\label{eqn:homogeneous_initial_value}
    \begin{split}
        &c_+(s) = \frac{1}{2}(\eta_0 - \I\varepsilon^{1/2}K(s)^{-1}\xi_0) \\
    &c_-(s) = \frac{1}{2}(\eta_0 + \I\varepsilon^{1/2}K(s)^{-1}\xi_0) {,}
    \end{split}
\end{equation}
which, together with~\eqref{eqn:bounds_c}, indicates the bounds 
\[
    |c_+(t)| \leq C(\eta_0 + \varepsilon^{1/2}\xi_0), \quad 
    |c_-(t)| \leq C(\eta_0 + \varepsilon^{1/2}\xi_0) {.}
\]
Plugging back into~\eqref{eqn:y} and~\eqref{eqn:py}, 
then we get the desired bound for $\Phi_{\ve}^{s,t}(\eta_0,\xi_0)$.
\end{proof}


We now return to the residual system (\ref{eqn:ODE_residue}).
By introducing the auxiliary variable $z_{\ve}:=\dot{y}_{\ve}$, this
system can be reformulated as a first-order system
\[
\left(\begin{array}{c}
\dot{y}_{\ve}\\
\dot{z}_{\ve}
\end{array}\right)=\left(\begin{array}{c}
z_{\ve}\\
-\ve^{-1}Ay_{\ve}
\end{array}\right)+\left(\begin{array}{c}
0\\
\psi_{\ve}(t)
\end{array}\right)
\]
Then by Duhamel's principle,
\[
\left(\begin{array}{c}
y_{\ve}(t)\\
z_{\ve}(t)
\end{array}\right)=\Phi_{\ve}^{0,t}\left(0,z_{0}\right)
+\int_{0}^{t}\Phi_{\ve}^{s,t}\left(0,\psi_{\ve}(s)\right)\,ds.
\]
 Now by Lemma \ref{lem:homogeneous}, 
the next lemma follows directly.\\

\begin{lem}\label{lem:inhomogeneous}
     Let $y_{\ve}$ be the solution to the residual system~\eqref{eqn:ODE_residue}. Then 
     \[
         \vert y_{\ve} \vert \leq C\ve^{1/2},\ \vert \dot{y}_{\ve} \vert \leq C
     \]
     on $[0,t_f]$.\\
\end{lem}
 

Now we are ready to complete the proof of the  estimate \eqref{eqn:coarse_estimate}.\\

\begin{proof}[Proof of Theorem~\ref{thm:main_theorem}(i)]
    Substituting
    $y_{\ve} = x_{\ve}- A(r_{\ve})^{-1}b(r_{\ve})$ 
    into Eq.~\eqref{eqn:0SCF_rp}, 
    the dynamics for $(\re,\pe)$ are given by 
    \begin{subequations}\label{eqn:dynamics_re}
    \begin{align}
        \dot{r}_{\ve} &= \pe \\
        \dot{p}_{\ve} &= G(\re,A(\re)^{-1}b(\re))+ e_{\ve} {,}
    \end{align}
    \end{subequations}
    where
    \begin{equation}\label{eqn:def_e}
    e_{\ve} = x_{\ve}^{\top} \left(\frac{\partial A}{\partial r}(r_{\ve})\right)y_{\ve}
     + \frac{1}{2} y_{\ve}^{\top} \left(\frac{\partial A}{\partial r}(r_{\ve})\right)y_{\ve} 
     - \frac{\partial b^{\top}}{\partial r}(r_{\ve})y_{\ve} {.}
    \end{equation}
    Note that Eq.~\eqref{eqn:dynamics_re} only differs from Eq.~\eqref{eqn:MD_rp} 
    by the extra term $e_{\ve}$.
    Then by the Alekseev-Gr\"obner lemma (cf., Theorem 14.5 of \cite{HairerNorsettWanner1987})
    \begin{equation}\label{eqn:Alekseev}
        \left(\begin{array}{c}
        r_{\ve}(t)\\
        p_{\ve}(t)
        \end{array}\right) = 
         \left(\begin{array}{c}
        r(t)\\
        p(t)
        \end{array}\right) + 
        \int_0^t \mathfrak{R}^{s,t}(r_{\ve}(s),p_{\ve}(s))\left(\begin{array}{c}
        0\\
        e_{\ve}(s)
        \end{array}\right) ds,
    \end{equation}
    where
    \[
    \mathfrak{R}^{s,t}(\eta,\xi) = [\partial_\eta \Psi^{s,t}(\eta,\xi),\ \partial_\xi \Psi^{s,t}(\eta,\xi) ],
    \]
    with $\Psi^{s,t}(\eta,\xi) \in \R^{2d}$ denoting the solution of Eq.~\eqref{eqn:MD_rp} 
    with starting time $s$ and 
    initial values $r(s) = \eta$, $p(s) = \xi$. Now the derivative of the solution of an ODE with respect to its initial condition can be obtained by solving an ODE (cf., Theorem 14.3 of \cite{HairerNorsettWanner1987}):
    \begin{align*}
        \frac{\partial}{\partial t} \mathfrak{R}^{s,t}(\eta,\xi) &= 
            \left(\begin{array}{cc}
                0 & I_d \\
                \frac{\partial h}{\partial r}(\Psi^{s,t}(\eta,\xi)) & 0
            \end{array}
            \right)
            \mathfrak{R}^{s,t}(\eta,\xi) \\
            \mathfrak{R}^{s,s}(\eta,\xi) &= I_{2d},
    \end{align*}
    where $h(r)$ represents the right hand side of Eq.~\eqref{eqn:MD_rp_b}. 
    By our system of ODEs satisfied by $\mathfrak{R}^{s,t}(\eta,\xi)$, together with Assumption~\ref{assump:tech} (including \emph{a priori} bounds for $r_{\ve}$ and $p_{\ve}$) and 
    Gr{\"o}nwall's inequality, we have that
    $\mathfrak{R}^{s,t}(r_{\ve}(s),p_{\ve}(s))$ is bounded independently of $\ve$ and $s\in [0,t_f]$. 
    Therefore~\eqref{eqn:Alekseev} implies
    \[
       \lvert \ret -
  r_{\star}\rvert,\  \lvert \pet - p_{\star}\rvert \leq C \sup_{t\in [0,t_f]}\vert e_{\ve}(t)\vert
    \]
    on $[0,t_f]$.
   Then the definition of $e_{\ve}$ (i.e., Eq.~\eqref{eqn:def_e})
    and Lemma~\ref{lem:inhomogeneous}, 
    together with the \emph{a priori} bounds for $r_{\ve}$ and $p_{\ve}$, 
    imply that
    \[
        \vert e_{\ve}(t)\vert \leq C \vert y_{\ve}(t) \vert \leq 
        C\ve^{1/2},
    \]
    where $C$ has been possibly enlarged in the second inequality, and 
    thus
  \begin{equation}
      \label{eqn:coareEB}
       \lvert \ret -
  r_{\star}\rvert,\  \lvert \pet - p_{\star}\rvert \leq C \ve^{1/2}
    \end{equation}
    on $[0,t_f]$.

    The error bound for $x_{\ve}$ can then be obtained as follows. First compute
    \begin{align*}
        \vert x_{\ve}(t) - x_\star (t) \vert &\leq 
        \vert x_{\ve}(t) - A(r_{\ve}(t))^{-1}b(r_{\ve}(t)) \vert 
        + \vert A(r_{\ve}(t))^{-1}b(r_{\ve}(t)) - 
        A(r(t))^{-1}b(r(t)) \vert\\
        &= \vert y_{\ve}(t) \vert 
        + \vert A(r_{\ve}(t))^{-1}b(r_{\ve}(t)) - 
        A(r(t))^{-1}b(r(t)) \vert \\
        &\leq C \ve^{1/2}
        + \vert f(r_\ve (t)) - f(r(t)) \vert,
    \end{align*}
    where we have used Lemma~\ref{lem:inhomogeneous} in the last inequality, and we have defined $f(r) := A(r)^{-1} b(r)$. Now since the eigenvalues of $A(r)$ are uniformly bounded away from zero, $f$ is a $C^1$ function. Together with the \emph{a priori} bounds on $r_{\ve}$ and $r_{\star}$, we have $\vert f(r_\ve (t)) - f(r(t)) \vert \leq C \vert r_\ve (t) - r_\star (t) \vert$ for $C$ independent of $t, \ve$.
    Then by~\eqref{eqn:coareEB}, the bound $|x_\ve - x_\star\vert \leq C \ve^{1/2}$ follows. 
\end{proof}

\section{Sharp error analysis for $d'=1$}\label{sec:analysis_1d}
We focus on the case when the dimension of the latent variable satisfies $d'=1$. We retain all definitions made above
for general $d$. Since $d'=1$, we denote $k_{_{\ve}}=K(r_{\ve})$ to
emphasize that this is a scalar quantity. Moreover, $U_{\ve,\pm}^s(t)=e^{\pm \I(\kappa_{\ve}(t)-\kappa_{\ve}(s))/\sqrt{\ve}}$,
where $\kappa_{\ve}(t)=\int_{0}^{t}k_{\ve}(s)\,ds$. Note that since
$k_{\ve}(t)\geq C^{-1}$ for all $t$, $\kappa$ is then strictly increasing
with $\dot{\kappa}_{\ve}(t)=k_{\ve}(t)\geq C^{-1}$. Then the inverse mapping $\kappa^{-1}$
is well-defined. Moreover, recall our uniform bounds (in $\ve$) over
$k_{\ve}(t)=K(r_{\ve}(t))$, as well as $\dot{k}_{\ve}$ and $\ddot{k}_{\ve}$
(following from bounds on $r_{\ve},\dot{r}_{\ve},\ddot{r}_{\ve}$),
from which we have in particular that $\vert\dot{\kappa}\vert,\vert\ddot{\kappa}\vert\leq C$.\\

\begin{lem} \label{lem:asymptoticsSharp1D}
Let $\Phi_{\ve}^{s,t}(\eta_0,\xi_{0})$ be
the flow map of the homogeneous system~\eqref{eqn:homogeneous}. 
Then 
 \[   \Phi^{s,t}(0,\xi_0) = \left(\begin{array}{c}
\ve^{1/2}k_{\ve}(t)^{-1/2}k_{\ve}(s)^{-1/2}\sin\left(\frac{\kappa_{\ve}(t)-\kappa_{\ve}(s)}{\sqrt{\ve}}\right)\xi_0\\
k_{\ve}(t)^{1/2}k_{\ve}(s)^{-1/2}\cos\left(\frac{\kappa_{\ve}(t)-\kappa_{\ve}(s)}{\sqrt{\ve}}\right)\xi_0
\end{array}\right) + 
\left(\begin{array}{c}
\Or(\ve)\\
\Or(\ve^{1/2})
\end{array}\right) {.}\\
\]
\end{lem}
\begin{proof}
The arguments used to prove this lemma are adapted from \cite{Tao2012}, where similar asymptotics are used to study Hermite polynomials.

As in the proof of Lemma~\ref{lem:homogeneous}, we omit dependence
on $\ve$ from the subscripts. 
Then we reproduce~\eqref{eqn:ODE_ab} from our proof of Lemma~\ref{lem:homogeneous} above with somewhat modified notation:
\begin{align*}
&\dot{c}_+=-\frac{1}{2}U_{+}^{-1}k^{-1}\dot{k}U_{+}c_++\frac{1}{2}U_{+}^{-1}k^{-1}\dot{k}U_{-}c_-,\\
&\dot{c}_-=\frac{1}{2}U_{-}^{-1}k^{-1}\dot{k}U_{+}c_+-\frac{1}{2}U_{-}^{-1}k^{-1}\dot{k}U_{-}c_-.
\end{align*}
Since $d'=1$, we can now commute operators to obtain 
\begin{align*}
&\dot{c}_+=-\frac{\dot{k}}{2k}c_+ + \frac{\dot{k}}{2k}e^{-2\I\kappa/\sqrt{\ve}}c_-,\\
&\dot{c}_-=-\frac{\dot{k}}{2k}c_-+\frac{\dot{k}}{2k}e^{2\I\kappa/\sqrt{\ve}}c_+.
\end{align*}

We introduce new variables $\gamma_+(t):=k(t)^{1/2}c_+(t)$ and $\gamma_-(t):=k(t)^{1/2}c_-(t)$. Note that
\begin{equation*}
\dot{\gamma}_+ = \,k^{1/2}\dot{c}_+ + \frac{\dot{k}}{2k^{1/2}}c_+
 =\,\frac{\dot{k}}{2k^{1/2}}e^{-2\I \kappa/\sqrt{\ve}}c_- {,}
\end{equation*}
 we have 
\[
\dot{\gamma}_+=\frac{\dot{k}}{2k}e^{-2\I\kappa/\sqrt{\ve}}\gamma_-, \quad
\dot{\gamma}_- = \frac{\dot{k}}{2k}e^{2\I\kappa/\sqrt{\ve}}\gamma_+.
\]
Recall our estimates (note that here we only focus on the case $\eta_0 = 0$)
\[
\vert c_+\vert,\vert c_-\vert \leq C\ve^{1/2}\vert\xi_{0}\vert
\]
from Lemma~\ref{lem:homogeneous}. It follows that
\[
\vert\gamma_+\vert,\vert\gamma_-\vert,\vert\dot{\gamma}_+\vert,\vert\dot{\gamma}_-\vert\leq C\ve^{1/2}\vert\xi_{0}\vert.
\]
The basic idea is that via ODEs for $\gamma_+,\gamma_-$, we know that $\gamma_+(t)-\gamma_+(s)$
can be written as an oscillatory integral of $\gamma_-$. Meanwhile,
our bounds on $\dot{\gamma}_-$ give us control over the oscillation
of $\gamma_-$, which guarantees some cancellation (corresponding to
a factor of $\sqrt{\ve}$) in the oscillatory integral. The same reasoning
applies with the roles of $\gamma_+$ and $\gamma_-$ exchanged. 

Now we carry out this argument. Write 
\begin{align*}
\gamma_+(t)-\gamma_+(s) & =\,\int_{s}^{t}\dot{\gamma}_+(\tau)\,d\tau\\
 & =\,\int_{s}^{t}\frac{\dot{k}(\tau)}{2k(\tau)}e^{-2\I\kappa(\tau)/\sqrt{\ve}}\gamma_-(\tau)\,d\tau\\
 & =\,\int_{\kappa(s)}^{\kappa(t)}\frac{\dot{k}(\kappa^{-1}(u))}{2k(\kappa^{-1}(u))}e^{-2\I u/\sqrt{\ve}}\gamma_-(\kappa^{-1}(u))\left[\kappa^{-1}\right]'(u)\,du.
\end{align*}
 Define 
\[
f(u):=\,\frac{\dot{k}(\kappa^{-1}(u))}{2k(\kappa^{-1}(u))}\left[\kappa^{-1}\right]^{'}(u)\gamma_-(\kappa^{-1}(u)).
\]
 By our previous discussion of uniform bounds, we have $\vert f\vert,\vert f'\vert\leq C\ve^{1/2}\vert\xi_{0}\vert$.
Then we rewrite our integral and integrate by parts: 
\begin{align*}
\gamma_+(t)-\gamma_+(s) & =\,\int_{\kappa(s)}^{\kappa(t)}f(u)e^{-2\I u/\sqrt{\ve}}\,du\\
 & =\,-\frac{\sqrt{\ve}}{2\I}\left(\left[f(u)e^{-2\I u/\sqrt{\ve}}\right]_{u=\kappa(s)}^{u=\kappa(t)}-\int_{\kappa(s)}^{\kappa(t)} f'(u)e^{-2\I u/\sqrt{\ve}}\,du\right),
\end{align*}
so 
\[
    \gamma_+(t) = \gamma_+(s) + \Or(\ve),
\]
i.e.,
\[
    k(t)^{1/2} c_+(t) = k(s)^{1/2} c_+(s) + \Or(\ve).
\]
Then by the uniform bound of $k(t)^{-1}$, we have
\[
    c_+(t) = k(t)^{-1/2}k(s)^{1/2}c_+(s) + \Or(\ve) 
    = -\frac{1}{2}\I \ve^{1/2}k(t)^{-1/2}k(s)^{-1/2}\xi_0 + \Or(\ve).
\]
A similar result holds for $\gamma_-$ by equivalent reasoning: 
\[
    c_-(t) = \frac{1}{2}\I \ve^{1/2}k(t)^{-1/2}k(s)^{-1/2}\xi_0 + \Or(\ve).
\]
Therefore by Lemma~\ref{lem:homogeneous} the flow map is given by
\begin{align*}
    \Phi^{s,t}(0,\xi_0) &= \left(\begin{array}{c}
e^{\I (\kappa(t)-\kappa(s))/\sqrt{\ve}}c_+(t) + e^{-\I (\kappa(t)-\kappa(s))/\sqrt{\ve}}c_-(t)\\
\I \ve^{-1/2} k(t) [e^{\I (\kappa(t)-\kappa(s))/\sqrt{\ve}}c_+(t) - e^{-\I (\kappa(t)-\kappa(s))/\sqrt{\ve}}c_-(t)]
\end{array}\right) \\
&= \left(\begin{array}{c}
\ve^{1/2}k(t)^{-1/2}k(s)^{-1/2}\sin\left(\frac{\kappa(t)-\kappa(s)}{\sqrt{\ve}}\right)\xi_0\\
k(t)^{1/2}k(s)^{-1/2}\cos\left(\frac{\kappa(t)-\kappa(s)}{\sqrt{\ve}}\right)\xi_0
\end{array}\right) + 
\left(\begin{array}{c}
\Or(\ve)\\
\Or(\ve^{1/2})
\end{array}\right) {.}
\end{align*}
\end{proof}

Now we turn again to the inhomogeneous residual system~\eqref{eqn:ODE_residue}.\\
\begin{lem}
\label{lem:inhomogeneousSharp1D} Let $y_{\ve}$ be the solution to
the residual system~\eqref{eqn:ODE_residue} with $y_{\ve}(0)=0$. Then
for $t\in[0,t_f]$,
\[
\left(\begin{array}{c}
y_{\ve}(t)\\
\dot{y}_{\ve}(t)
\end{array}\right)=\left(\begin{array}{c}
\ve^{1/2}k(t)^{-1/2}k(0)^{-1/2}\sin\left(\frac{\kappa(t)}{\sqrt{\ve}}\right)z_0\\
k(t)^{1/2}k(0)^{-1/2}\cos\left(\frac{\kappa(t)}{\sqrt{\ve}}\right)z_0
\end{array}\right)+\left(\begin{array}{c}
\Or (\ve)\\
\Or (\ve^{1/2})
\end{array}\right).
\]
\end{lem}

\begin{proof}
Recall that by introducing the auxiliary variable $z_{\ve}:=\dot{y}_{\ve}$,
this system can be reformulated as the first-order system 
\[
\left(\begin{array}{c}
\dot{y}_{\ve}(t)\\
\dot{z}_{\ve}(t)
\end{array}\right)=\left(\begin{array}{c}
z_{\ve}(t)\\
-\ve^{-1}A(t)y_{\ve}(t)
\end{array}\right)+\left(\begin{array}{c}
0\\
\psi_{\ve}(t)
\end{array}\right),
\]
 and Duhamel's principle yields 
\[
\left(\begin{array}{c}
y_{\ve}(t)\\
z_{\ve}(t)
\end{array}\right)=\Phi_{\ve}^{0,t}\left(0,z_{0}\right)+\int_{0}^{t}\Phi_{\ve}^{s,t}\left(0,\psi_{\ve}(s)\right)\,ds.
\]
 Thus, by Lemma~\ref{lem:asymptoticsSharp1D}, it suffices to show that 
\[
\int_{0}^{t}\Phi_{\ve}^{s,t}\left(0,\psi_{\ve}(s)\right)\,ds=\left(\begin{array}{c}
\Or (\ve)\\
\Or (\ve^{1/2})
\end{array}\right).
\]
 Also by Lemma~\ref{lem:asymptoticsSharp1D} we have that 
\[
\Phi_{\ve}^{s,t}(0,\psi_{\ve}(s))=\left(\begin{array}{c}
\ve^{1/2}k(t)^{-1/2}k(s)^{-1/2}\sin\left(\frac{\kappa(t)-\kappa(s)}{\sqrt{\ve}}\right)\psi_{\ve}(s)\\
k(t)^{1/2}k(s)^{-1/2}\cos\left(\frac{\kappa(t)-\kappa(s)}{\sqrt{\ve}}\right)\psi_{\ve}(s)
\end{array}\right) + 
\left(\begin{array}{c}
\Or(\ve)\\
\Or(\ve^{1/2})
\end{array}\right).
\]
 Thus it suffices to show that 
\[
I_{\ve}:=\int_{0}^{t}\sin\left(\frac{\kappa(t)-\kappa(s)}{\sqrt{\ve}}\right)k(s)^{-1/2}\psi(s)\,ds= \Or (\ve^{1/2})
\]
 and 
\[
J_{\ve}:=\int_{0}^{t}\cos\left(\frac{\kappa(t)-\kappa(s)}{\sqrt{\ve}}\right)k(s)^{-1/2}\psi(s)\,ds= \Or (\ve^{1/2}).
\]
 
 Changing variables by $u=\kappa(s)$ we obtain 
\[
I_{\ve}=\int_{0}^{\kappa(t)}\sin\left(\frac{\kappa(t)-u}{\sqrt{\ve}}\right)\left[\kappa^{-1}\right]'(u)k(\kappa^{-1}(u))^{-1/2}\psi(\kappa^{-1}(u))\,du.
\]
 Now define 
\[
g(u):=\left[\kappa^{-1}\right]'(u)k(\kappa^{-1}(u))^{-1/2}\psi(\kappa^{-1}(u)).
\]
As in the argument in Lemma~\ref{lem:asymptoticsSharp1D}, we will
need that $\vert g\vert,\vert \dot{g} \vert\leq C$ uniformly
in $\ve$. This could be guaranteed if we knew that $\vert\dot{\psi}\vert\leq C$.
(We have already seen that $\vert\psi\vert$ is uniformly bounded.)
Recall that 
\[
\psi(t)=-\frac{d^{2}}{dt^{2}}\left[A(r_{\ve}(t))^{-1}b(r_{\ve}(t))\right],
\]
 so by the $C^{3}$ assumption on $A,b$, it will suffice to show
a uniform bound on $\big\vert \frac{\ud^3 {r}_{\ve}}{\ud t^3}\big\vert$. Now differentiating
the XLMD system~\eqref{eqn:0SCF_rp} 
we see that it then suffices
to obtain a uniform bound on $\vert\dot{x}_{\ve}\vert$. But then
it suffices to obtain a uniform bound on $\vert\dot{y}_{\ve}\vert$,
since $y_{\ve}=x_{\ve}-A^{-1}(r_{\ve})b(r_{\ve})$. Indeed, such a
bound has already been obtained (Lemma~\ref{lem:inhomogeneous}).
Then in conclusion, $\vert g\vert,\vert\dot{g}\vert\leq C$
uniformly in $\ve$, as desired.

Now rewrite the integral for $I_{\ve}$ and integrate by parts: 
\begin{align*}
I_{\ve} & =\,\int_{0}^{\kappa(t)}\sin\left(\frac{\kappa(t)-u}{\sqrt{\ve}}\right)g(u)\,du\\
 & =\,-\ve^{1/2}\left(\left[\cos\left(\frac{\kappa(t)-u}{\sqrt{\ve}}\right)g(u)\right]_{u=0}^{u=\kappa(t)}-\int_{0}^{\kappa(t)}\cos\left(\frac{\kappa(t)-u}{\sqrt{\ve}}\right)g'(u)\,du\right),
\end{align*}
 from which it follows that $I_{\ve}=\Or(\ve^{1/2})$. The result $J_{\ve}=\Or(\ve^{1/2})$
is obtained similarly. This finishes the proof of the lemma. 
\end{proof}

\smallskip

\begin{rem}
Observe that $y_{\ve}$ is in fact $\Or(\ve)$
in the case of optimally compatible initial condition, i.e., $z_{0}=0$. Then in this case, to establish the $\Or(\ve)$ errors in $(r,p)$, we may follow the idea of coarse estimate in Section 3.2  and apply the theorem of Alekseev and Gr\"obner.
However, we present a more general proof below that encompasses both the compatible and the optimally compatible intial conditions. We have obtained a very precise understanding of the oscillatory nature of $y_{\ve}$---in
fact, an explicit formula up to an error of order $\Or(\ve)$---and
it is this that we use to show that it in fact only yields an
error of $\Or(\ve)$, even in the case of non-optimally-compatible $z_0$.\\
\end{rem}

\smallskip

\begin{proof}[Completion of the proof for the sharp estimate \eqref{eqn:sharp_estimate}]
Recall the XLMD system~\eqref{eqn:0SCF_rp}: 
\begin{align*}
\ddot{r}_{\ve}&=G(\re,\xe)\\
\ve\ddot{x}_{\ve}&=b(r_{\ve})-A(r_{\ve})x_{\ve}
\end{align*}
and the exact MD~\eqref{eqn:MD_rp}
\[
\ddot{r}_{\star}=G\left(r_{\star},A(r_{\star})^{-1}b(r_{\star})\right).
\]
Now we already know that $\vert y_{\ve}\vert\leq C\ve^{1/2}$ by Lemma~\ref{lem:inhomogeneous}, and
moreover $x_{\ve}=y_{\ve}+A(r_{\ve})^{-1}b(r_{\ve})$, so it follows
that $\vert x_{\ve}-A(r_{\ve})^{-1}b(r_{\ve})\vert\leq C\ve^{1/2}$.
Moreover, we know that $\vert r_{\ve}-r_{\star}\vert\leq C\ve^{1/2}$ as well from the 
coarse estimate, 
so  $\vert x_{\ve}-A(r_{\star})^{-1}b(r_{\star})\vert\leq C\ve^{1/2}$. Then
by the Taylor expansion of $G(\re,\xe)$ around $(r_{\star},A(r_{\star})^{-1}b(r_{\star}))$, 
it follows that 
\begin{align*}
\ddot{r}_{\ve}= &G(r_{\star},A(r_{\star})^{-1}b(r_{\star}))
+\left[\frac{\partial G}{\partial r}\left(r_{\star},A(r_{\star})^{-1}b(r_{\star})\right)\right][r_{\ve}-r_{\star}] \\
&\quad + \left[\frac{\partial G}{\partial x}\left(r_{\star},A(r_{\star})^{-1}b(r_{\star})\right)\right][\xe-A(r_{\star})^{-1}b(r_{\star})]+\Or(\ve).
\end{align*}
A further Taylor expansion tells that 
\begin{align*}
  \xe-A(r_{\star})^{-1}b(r_{\star}) &= \xe 
  - \left[A(\re)^{-1}b(\re) - \frac{\partial (A^{-1}b)}{\partial r}(r_{\star})[\re-r_{\star}] + \Or(\ve) \right] \\
  &= \ye + \frac{\partial (A^{-1}b)}{\partial r}(r_{\star})[\re-r_{\star}] + \Or(\ve){.}
\end{align*}
Then
\begin{align*}
\ddot{r}_{\ve}= &G(r_{\star},A(r_{\star})^{-1}b(r_{\star})) \\
&\quad +\left[\frac{\partial G}{\partial r}\left(r_{\star},A(r_{\star})^{-1}b(r_{\star})\right) + \frac{\partial G}{\partial x}\left(r_{\star},A(r_{\star})^{-1}b(r_{\star})\right)\frac{\partial (A^{-1}b)}{\partial r}(r_{\star}) \right][r_{\ve}-r_{\star}] \\
&\quad + \left[\frac{\partial G}{\partial x}\left(r_{\star},A(r_{\star})^{-1}b(r_{\star})\right)\right]\ye +\Or(\ve).
\end{align*}
Define 
\begin{align*}
\Upsilon(t)&:=\frac{\partial G}{\partial r}\left(r_{\star},A(r_{\star})^{-1}b(r_{\star})\right) + \frac{\partial G}{\partial x}\left(r_{\star},A(r_{\star})^{-1}b(r_{\star})\right)\frac{\partial (A^{-1}b)}{\partial r}(r_{\star}), \\
\Gamma(t)&:=\frac{\partial G}{\partial x}\left(r_{\star},A(r_{\star})^{-1}b(r_{\star})\right).
\end{align*}
 Note that $\Upsilon$ and $\Gamma$ do not depend on the parameter
$\ve$. Then we can write the dynamics for $r_{\ve}$ more simply
as 
\[
\ddot{r}_{\ve}=G(r_{\star},A(r_{\star})^{-1}b
(r_{\star})) + \Upsilon[r_{\ve}-r_{\star}]+\Gamma y_{\ve} + \Or(\ve).
\]
Define a new variable $\theta_{\ve}:=r_{\ve}-r_{\star}$, which measures
the error in the $r$ variable. Subtracting the ODEs for $r_{\ve}$
and $r_{\star}$ we obtain 
\[
\ddot{\theta}_{\ve}=\Upsilon\theta_{\ve}+\Gamma y_{\ve}+\Or(\ve).
\]
 Of course, since we have chosen $r_{\ve}(0)=r_{\star}(0)$ and $\dot{r}_{\ve}(0)=\dot{r}_{\star}(0)$,
we have the initial conditions $\theta_{\ve}(0)=0$, $\dot{\theta}_{\ve}(0)=0$.


We view the ODE for $\theta_{\ve}$ as a perturbation of the homogeneous 
ODE 
\[\ddot{\widetilde{\theta}}_{\ve}=\Upsilon\widetilde{\theta}_{\ve}.\]
The solution of this homogeneous ODE can be given as 
\begin{equation*}
   \left( \begin{array}{c}
       \widetilde{\theta}(t)  \\
       \dot{\widetilde{\theta}}(t) 
    \end{array}\right) = \Theta^s(t) 
    \left(\begin{array}{c}
       \widetilde{\theta}(s)  \\
       \dot{\widetilde{\theta}}(s) 
    \end{array}\right) 
\end{equation*}
where 
\[
    \Theta^s(t) = \mathcal{T} \exp \left(\int_s^t\left(\begin{array}{cc}
       0  & 1 \\
        \Upsilon(\tau) & 0
    \end{array}\right)d\tau \right).
\]
Here we have used the time ordering notation introduced earlier.
Since $t\mapsto\Upsilon(t)$ is $C^{1}$, $\Theta^s(t)$ is $C^{1}$ in both $t$ and $s$ (cf., Theorems 14.3, 14.4 of \cite{HairerNorsettWanner1987}). 
Then by Duhamel's principle, we have 
\begin{align*}
\left(\begin{array}{c}
\theta(t)\\
\dot{\theta}(t)
\end{array}\right) & =\,\int_{0}^{t}\Theta^s(t)\left(\begin{array}{c}
    0 \\
    \Gamma(s)\ye(s) 
\end{array}\right)\,ds+\Or(\ve)\\
 & =\,\int_{0}^{t}\Theta^s(t)\left(\begin{array}{c}
    0 \\
    \ve^{1/2}\Gamma(s)k_{\ve}(s)^{-1/2}k_{\ve}(0)^{-1/2}\sin\left(\frac{\kappa_{\ve}(s)}{\sqrt{\ve}}\right)z_0 
\end{array}\right)\,ds+\Or(\ve)\\
 & =\,\ve^{1/2}k_{\ve}(0)^{-1/2}z_{0}\int_{0}^{t}\Gamma(s)k_{\ve}(s)^{-1/2}\Theta^s(t)\left(\begin{array}{c}
     0  \\
     \sin\left(\frac{\kappa_{\ve}(s)}{\sqrt{\ve}}\right) 
 \end{array}\right)\,ds+\Or(\ve),
\end{align*}
 where in the last two steps we have used Lemma~\ref{lem:inhomogeneousSharp1D}.
So we have reduced our problem to showing that the oscillatory integral 
\[
\int_{0}^{t}\Gamma(s)k_{\ve}(s)^{-1/2}\Theta^s(t)\left(\begin{array}{c}
     0  \\
     \sin\left(\frac{\kappa_{\ve}(s)}{\sqrt{\ve}}\right) 
 \end{array}\right)\,ds
\]
 is $\Or(\ve^{1/2})$, where the product in the integrand is a matrix-vector
multiplication. This is a key difference from the estimate for general $d'$.

Note that $\Gamma(s),k_{\ve}(s),\Theta^s(t)$ are all $C^1$ in $s$ and 
bounded uniformly on $\ve$, we can employ
the integration-by-parts argument used for oscillatory integrals above,
i.e., we can rewrite the integral as 
\[
-\ve^{1/2}\left(\left[\Gamma(t)k_{\ve}(t)^{-1/2}\Theta^t(t)\left(\begin{array}{c}
0\\
\cos\left(\frac{\kappa_{\ve}(s)}{\sqrt{\ve}}\right)
\end{array}\right)\right]-\int_{0}^{t}\frac{d\left[\Gamma(s)k_{\ve}(s)^{-1/2}\Theta^s(t)\right]}{ds}\left(\begin{array}{c}
0\\
\cos\left(\frac{\kappa_{\ve}(s)}{\sqrt{\ve}}\right)
\end{array}\right)\,ds\right),
\]
which is evidently $\Or (\ve^{1/2})$.

This completes the proof that $r_\ve - r_\star, p_\ve - p_\star = \Or(\ve)$. To conclude the proof of Theorem~\ref{thm:main_theorem}(ii) we obtain the error bound for $x_\ve$ by essentially copying the argument at the end of the proof of Theorem~\ref{thm:main_theorem}(i). To wit, we recall from said argument that
\[
 \vert x_{\ve}(t) - x_\star (t) \vert \leq 
 \vert y_{\ve}(t) \vert + 
 C \vert r_\ve (t) - r_\star (t) \vert
\]
for $t\in [0,t_f]$, where $C$ is independent of $t, \ve$. But we just showed that the second term on the right-hand side of this inequality is $\Or (\ve)$. Meanwhile, by Lemma~\ref{lem:inhomogeneousSharp1D} we have that $y_\ve (t) = \Or( \ve^{1/2})$ in the general case of compatible initial condition and $y_\ve (t) = \Or( \ve)$ in the case of optimally compatible initial condition. Hence $x_\ve - x_\star$ is $ \Or(\ve^{1/2})$ in the former case and $ \Or(\ve)$ in the latter. This explains why the error of the latent variable differs between the compatible and optimally compatible cases, in spite of the fact that $\Or (\ve)$ error is achieved by $r_\ve, p_\ve$ in both cases.
\end{proof}
 
\section{Numerical results}\label{sec:numer}

We study the convergence order of XLMD under different initial conditions for the auxiliary variable $x$, using a toy model with 
$$
    U(r) = \frac{1}{4} |r|^4 + \cos\left(2\sum_{j=1}^3 r_j\right), \quad r = (r_1,r_2,r_3)^{\top} \in \RR^3,
$$
Here $A(r)$ a sparse matrix in $\RR^{20\times 20}$ with non-zero entries
\begin{align*}
    &A_{k,k}(r) = 2 + |r|^2, \quad 1 \leq k \leq 20,\\
    &A_{k,k+1}(r) = A_{k+1,k}(r) = -1, \quad 1 \leq k \leq 19,  \\
    &A_{k,k+2}(r) = A_{k+2,k}(r) = \frac{1}{2}(1-|r|^2), \quad 1 \leq k \leq 18, 
\end{align*}
and we define $b(r)\in\RR^{20}$ by 
$$ b_k(r) = \sin\left(\frac{k}{10} r_1 + \left(1-\frac{k}{20}\right)r_2 + r_3\right), \quad \forall 1 \leq k \leq 20. $$
The exact dynamics are initialized with conditions 
$$r_{\star}(0) = (0,0.5,1)^{\top}, \quad p_{\star}(0) = (1,0.5,-1)^{\top}. $$
The Verlet scheme~\cite{Verlet1967} is used to propagate both the exact dynamics and the XLMD. 
The time step size is fixed to be $10^{-5}$, 
and the time interval is fixed to be $[0,5]$. 

For XLMD, we initialize the dynamics with 
$$\re(0) = r_{\star}(0), \quad \pe(0) = p_{\star}(0), $$
and we consider three types of initial condition for the auxiliary variables.
\begin{itemize}
    \item Optimally compatible initial condition: 
    $$\xe(0) = x_{\star}(0), \quad \dxe(0) = \dot{x}_{\star}(0),$$
    computed via Eq.~\eqref{eqn:compatible_initial_value} and~\eqref{eqn:optimally_compatible_initial_value},
    \item Compatible initial condition: 
    $$\xe(0) = x_{\star}(0), \quad \dxe(0) = (0,\cdots,0)^{\top}.$$
    \item Incompatible initial condition: 
    $$\xe(0) = x_{\star}(0) + \frac{1}{2}(1,-1,1,-1,\cdots,1,-1)^{\top}, \quad \dxe(0) = (0,\cdots,0)^{\top}.$$
\end{itemize}
We perform the time propagation for each choice until the same final time and then measure the errors by computing $\max_{t\in[0,5]}\|\re(t)-r_{\star}(t)\|_2,\max_{t\in[0,5]}\|\pe(t)-p_{\star}(t)\|_2$, and $\max_{t\in[0,5]}\|\xe(t)-x_{\star}(t)\|_2$. 

\begin{figure}
    \centering
    \includegraphics[width=0.7\linewidth]{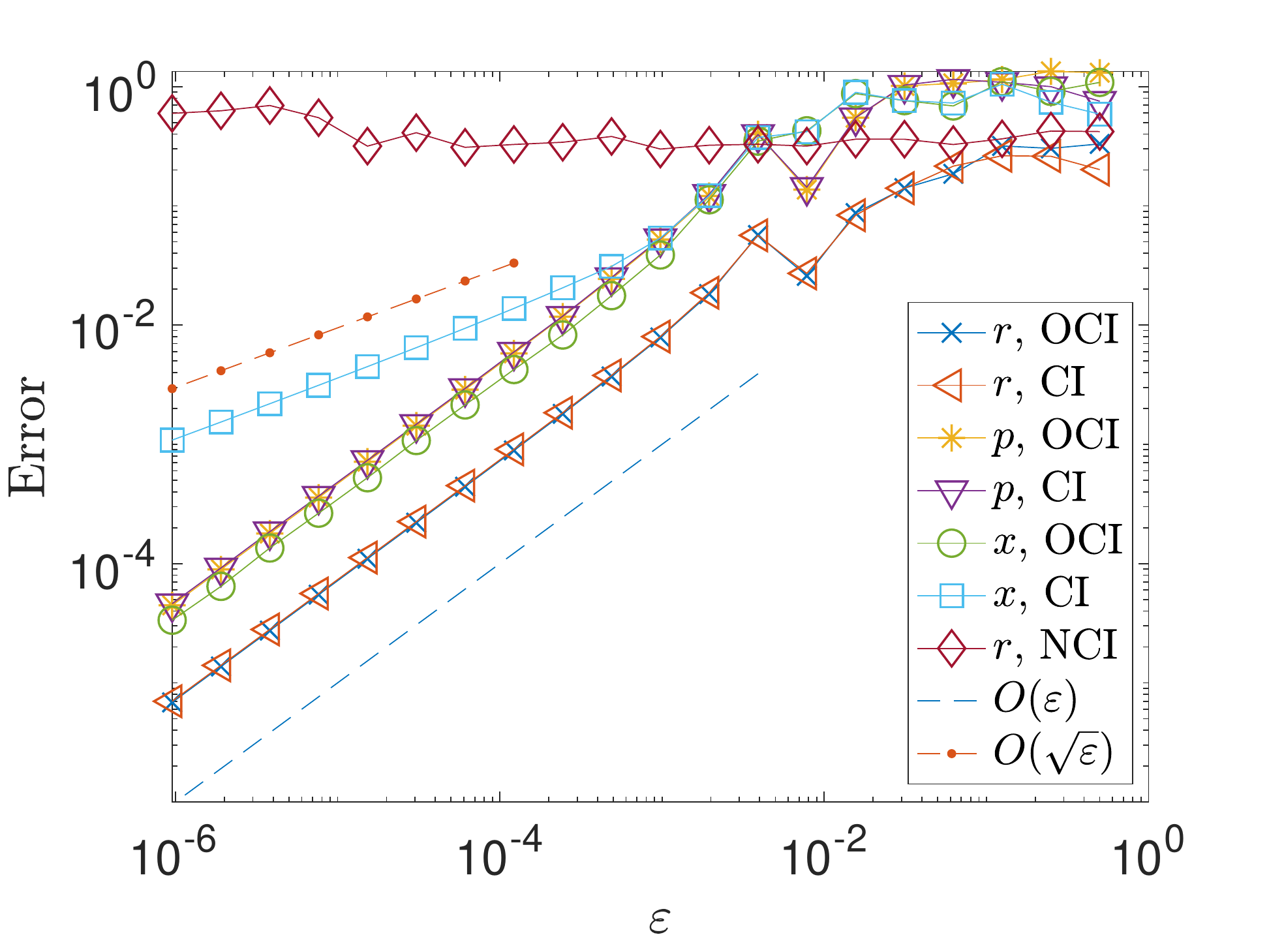}
    \caption{Numerical error versus $\ve$ for different initial conditions. In the legend, $r,p,x$ indicate the variable for which the error is measured. Moreover OCI is short for the optimally compatible initial condition, CI for the compatible value condition, and NCI for the incompatible initial condition.}
    \label{fig:0SCF_ini}
\end{figure}
\begin{table}[]
    \centering
    \begin{tabular}{|c|ccc|}\hline
         & \multicolumn{3}{c|}{Estimated order of convergence} \\
         & $r$ & $p$ & $x$ \\\hline
        Optimally compatible & 1.0067 & 1.0076 & 1.0021 \\
        Compatible & 1.0066 & 1.0055 & 0.5351  \\\hline
    \end{tabular}
    \caption{Numerically estimated order of convergence of XLMD for differential initial conditions. }
    \label{tab:0SCF_scaling}
\end{table}

The errors of the variables $r,p,x$ under different initial conditions and different choices of $\ve$ are shown in Figure \cref{fig:0SCF_ini}, and numerical estimates of asymptotic error scale are shown in Table \cref{tab:0SCF_scaling}. 
As $\ve$ goes to 0, both optimally compatible and compatible initial conditions yield greater accuracy, while there is no convergence if the initial condition of $x$ is incompatible. 
This shows that XLMD is only effective when the auxiliary system is correctly initialized. 
The convergence orders of $r$ and $p$ are 1 for both optimally compatible and compatible initial conditions, while the optimally compatible initial condition allows for better convergence in $x$ than does the compatible initial condition. 

Compared to our main theoretical result \cref{thm:main_theorem}, we find that the error bounds that we obtained for the setting of $d'=1$ are sharp in all cases, even though $d'>1$. Our analysis for general $d'$ is in fact sharp for the error in $x$ when the initial condition is only compatible. However, it is not sharp for the errors in $r$ and $p$, nor for the error in $x$ in the case of optimally compatible initial condition.

\section*{Acknowledgments} 

This work was partially supported by the National Science Foundation
under grant DMS-1652330 (D.A. and L.L.), by the Department of Energy under grant DE-SC0017867 and the U.S. Department of Energy, Office of Science, Office of Advanced Scientific Computing Research, Scientific Discovery through Advanced Computing (SciDAC) program (L.L.), and by the National Science Foundation Graduate Research Fellowship Program under grant DGE-1106400 and the National Science Foundation under Award No. 1903031 (M.L.). We thank Teresa Head-Gordon for helpful discussions.

\bibliographystyle{siam}
\bibliography{zeroscf}

\appendix
\label{sec:appendix}

\section{Proof of \cref{prop:solution}}

\begin{proof}
    The proposition is proved as follows: first we establish the existence and uniqueness of the solution on a neighborhood of $0$ by referring to standard theorems, then we prove the desired \emph{a priori} bounds for the solution on this neighborhood. The global existence and uniqueness on the entire time interval, as well as the bounds, can be then established by an extension theorem. 
    Throughout the proof $\varepsilon$ is viewed as a fixed positive parameter.

    By introducing $z_{\ve} = \sqrt{\ve}\dxe$, we first rewrite the exact MD as a first-order system of differential equations
    \begin{equation}
    \begin{split}
        \dot{r}_{\star}(t) &= p_{\star}(t), \\
         \dot{p}_{\star}(t) &= F(r_{\star}(t)) - \left[\frac{1}{2}b^{\top} A^{-1}\frac{\partial A}{\partial r}A^{-1}b - \frac{\partial b^{\top}}{\partial r}A^{-1}b\right](r_{\star}(t)), 
    \end{split}
    \end{equation}
    and XLMD as
    \begin{equation}
        \begin{split}
            \dre &= \pe,\\
            \dpe &= F(\re) - \frac{\partial Q}{\partial r}(\re,\xe),\\
            \dxe &= \frac{1}{\sqrt{\ve}}z_{\ve}, \\
            \dot{z}_{\ve} &= \frac{1}{\sqrt{\ve}}(b(\re)-A(\re)\xe).
        \end{split}
    \end{equation}
    By~\cite[Theorem 1.2 and 1.3]{Chicone2006}, there exists $\delta > 0$ (which might depend on $\ve$ for XLMD) such that there exist unique solutions $r_{\star}$ and $(\re,\xe)$ of MD and XLMD, respectively, on the interval $(0,\delta)$, and moreover $r_{\star}$, $\re$ and $\xe$ are $C^3$ functions. 

    Now we establish the uniform bounds of the solutions on the interval $(0,\delta)$. 
    For Eq.~\eqref{eqn:General_Dynamics}, consider the energy which is defined as 
    \[
    E_{\star}(t) = \frac{1}{2}\vert p_{\star}\vert^{2}+U(r_{\star})-\frac{1}{2}b(r_{\star})^{\top}A(r_{\star})^{-1}b(r_{\star}).
    \]
    Note that $\dot{E}_{\star}(t) = 0$
    and thus $E_{\star}(t) = E_{\star}(0)$ for all $t \in (0,\delta)$. In particular, $E(t)$ is bounded on this interval. 
    By Assumption~\ref{assump:tech},
    \[b(r_{\star})^{\top}A(r_{\star})^{-1}b(r_{\star}) \leq \frac{1}{C} |b(r_{\star})|^2{.}
    \] 
    Together with the assumptions that $b$ is bounded and $U$ is bounded from 
    below, we deduce that $U(r_{\star})-\frac{1}{2}b(r_{\star})^{\top}A(r_{\star})^{-1}b(r_{\star})$ is bounded from below. 
    Therefore $\frac{1}{2}\vert p_{\star}\vert^{2}$ is bounded from above, indicating 
    that $p_{\star} = \Or(1)$.
    After integration, the bound for $ r_{\star} = \Or(1)$ is immediately obtained. The bound for $\ddot{r}_{\star}$ can be obtained by 
    directly plugging the bound for $r_{\star}$ back into Eq.~\eqref{eqn:General_Dynamics}. 
    
    For Eq.~\eqref{eqn:General_Deterministic}, 
    there also exists a conserved energy $E_{\ve}(t) = 
    E_{\ve}(0)$, defined by
    \[
        E_{\ve}(t) =
        \frac{1}{2}\vert\pe\vert^{2} 
        + \frac{1}{2}\ve \vert\dot{\xe}\vert^{2} + U(\re) 
        + \frac{1}{2}\xe^{\top}A(\re)\xe - b(\re)^{\top}\xe.
    \]
    Again, by the uniformly positive definite property of $A$ and the
    uniform bound on $b$, 
    the interaction energy $\frac{1}{2}\xe^{\top}A(\re)\xe - b(\re)^{\top}\xe$
    is bounded from below. 
    Together with the assumption that $U$ is bounded from below, this implies 
    that $\frac{1}{2}\vert p_{\star}\vert^{2}$ and $\frac{1}{2}\ve \vert\dot{\xe}\vert^{2}$ are bounded from above, 
    indicating $\sqrt{\ve}\dot{x}_{\ve} = \Or(1)$ and $p_{\star} = \Or(1)$, from which it follows by integration that $r_{\star} = \Or(1)$.
    To obtain the uniform bound for $\ddot{r}_\ve$, 
    it is sufficient, based on Eq.~\eqref{eqn:General_Deterministic},
    to obtain a uniform bound for $\xe$. 
    This can be done via the energy $E_{\ve}(t)$ again. 
    Notice that the first three terms are all bounded from below, 
    so the sum of last two terms are bounded from above, which indicates that 
    \[
        \frac{1}{2C}|\xe|^2 - \sup_r\{|b(r)|\}\,|\xe| \leq \frac{1}{2}\xe^{\top}A(\re)\xe - b(\re)^{\top}\xe = \Or(1) {,}
    \]
    and thus $\xe = \Or(1)$.
    
    Finally, note that our derivation shows that our \emph{a priori} bounds hold on \emph{any} interval on which the MD and XLMD solutions exist. Hence an extension result~\cite[Corollary I-3-4]{HsiehSibuya1999} ensures the global existence and uniqueness of the solutions on the time interval $[0,t_f]$. 
\end{proof}

\end{document}